\documentclass[12pt,reqno]{amsart}

\usepackage{amsmath,amssymb,amscd,latexsym,amsthm,tabularx,array}
\theoremstyle{theorem}
\newtheorem{theorem}{Theorem}

\theoremstyle{theorem}
\newtheorem{corollary}[theorem]{Corollary}

\theoremstyle{theorem}
\newtheorem{lemma}[theorem]{Lemma}

\theoremstyle{definition}

\begin{document}

\title[On a class of diagonal equations]{On a class of diagonal equations\\ over finite fields}

\author{Ioulia N. Baoulina}

\dedicatory{To the memory of my first teacher in number theory, Elena B. Gladkova (1953~--~2015)}

\address{Department of Mathematics, Moscow State Pedagogical University, Krasnoprudnaya str. 14, Moscow 107140, Russia}
\email{jbaulina@mail.ru}

\date{}

\maketitle

\begin{abstract}
Using properties of Gauss and Jacobi sums, we derive explicit formulas for the number of solutions to a diagonal equation of the form $x_1^{2^m}+\dots+x_n^{2^m}\!=0$ over a finite field of characteristic $p\equiv\pm 3\!\pmod{8}$.  All of the evaluations are effected in terms of parameters occurring in quadratic partitions of some powers of~$p$.
\end{abstract}

\keywords{{\it Keywords}: Equation over a finite field; diagonal equation; Gauss sum; Jacobi sum.}

\subjclass{Mathematics Subject Classification 2010: 11G25, 11T24}

\thispagestyle{empty}

\section{Introduction}

Let $\mathbb F_q$ be a finite field of characteristic~$p>2$ with $q=p^s$ elements,  $\eta$ be the quadratic character on $\mathbb F_q$  ($\eta(x)=+1, -1, 0$ according as $x$ is a square, a non-square or zero in $\mathbb F_q$), and $\mathbb F_q^*=\mathbb F_q^{}\setminus\{0\}$. A diagonal equation over $\mathbb F_q$ is an equation of the type
\begin{equation}
\label{eq1}
a_1^{}x_1^{d_1}+\dots+a_n^{}x_n^{d_n}=b,
\end{equation}
where $a_1,\dots,a_n\in\mathbb F_q^*$, $b\in\mathbb F_q$ and $d_1,\dots,d_n$ are positive integers.  As $x_j$ runs through all elements of $\mathbb F_q$, $x_j^{d_j}$ runs through the same elements as $x_j^{\gcd(d_j,q-1)}$ does with the same multiplicity. Therefore, without loss of generality, we may assume that $d_j$ divides $q-1$ for all $j$. Denote by  $N[a_1^{}x_1^{d_1}+\dots+a_n^{}x_n^{d_n}=b]$ the number of solutions to \eqref{eq1} in $\mathbb F_q^n$.

The pioneering work on diagonal equations has been done by Weil~\cite{W}, who expressed the number of solutions in terms of Gauss sums. For certain choices of coefficients $a_1,\dots,a_n,b$, exponents $d_1,\dots,d_n$ and finite fields $\mathbb F_q$, the explicit formulas for the number of solutions can be deduced from Weil's expression, see \cite{B3, BEW, CS, J, LN, MP, S, SY, W1, W2} for some results in this direction. However, in general, it is a difficult task to determine $N[a_1^{}x_1^{d_1}+\dots+a_n^{}x_n^{d_n}=b]$.

In this paper, we consider a diagonal equation of the form
\begin{equation}
\label{eq2}
x_1^{2^m}+\dots+x_n^{2^m}=0,
\end{equation}
where $m$ is a positive integer with $2^m\mid(q-1)$. It is well known (see \cite[Theorem~10.5.1]{BEW} or \cite[Theorems~6.26 and 6.27]{LN}) that
$$
N[x_1^2+\dots+x_n^2=0]=\begin{cases}
q^{n-1}+\eta((-1)^{n/2})q^{(n-2)/2}(q-1)&\text{if $n$ is even,}\\
q^{n-1}&\text{if $n$ is odd.}
\end{cases}
$$
Moreover, if $p\equiv 3\!\pmod{4}$ and $2\mid s$, then it follows from the result of Wolfmann~\cite[Corollary~4]{W1} that
$$
N[x_1^4+\dots+x_n^4=0]=q^{n-1}+(-1)^{((s/2)-1)n}q^{(n-2)/2}(q-1)\cdot\frac{3^n+(-1)^n\cdot 3}4.
$$
Further, for any $m$ with $2^m\mid(q-1)$, it is not hard to show that
$$
N[x_1^{2^m}+x_2^{2^m}=0]=\begin{cases}
2^m(q-1)+1&\text{if $2^{m+1}\mid(q-1)$,}\\
1&\text{if $2^m\parallel(q-1)$.}
\end{cases}
$$

The goal of this paper is to determine explicitly $N[x_1^{2^m}+\dots+x_n^{2^m}=0]$ for an arbitrary $n$ in the case when $p\equiv\pm 3\!\pmod{8}$ and
$$
m\ge\begin{cases}
3&\text{if $p\equiv\hphantom{-}3\!\!\pmod{8}$,}\\
2&\text{if $p\equiv -3\!\!\pmod{8}$.}
\end{cases}
$$
In Section~\ref{s3}, we treat the case $p\equiv 3\!\pmod{8}$. The main results of this section are Theorems~\ref{t1} and \ref{t2}, in which we cover the cases $2^{m+1}\mid(q-1)$ and $2^m\parallel(q-1)$, respectively. Our main results in Section~\ref{s4} are Theorems~\ref{t3} and \ref{t4}, in which we deal with the case $p\equiv-3\!\pmod{8}$. All of the evaluations in Sections \ref{s3} and \ref{s4} are effected in terms of parameters occurring in quadratic partitions of some powers of~$p$. The results of numerical experiments are presented in Section~\ref{s5}. Applications of our results to some other diagonal equations are discussed in Section~\ref{s6}.

\section{Preliminary lemmas}
\label{s2}

Let, as usual, $\zeta_k=\exp(2\pi i/k)$. Let $\psi$ be a nontrivial character on $\mathbb F_q$. We extend $\psi$ to all of $\mathbb F_q$ by setting $\psi(0)=0$. The Gauss sum $G(\psi)$ over $\mathbb F_q$ is defined by
$$
G(\psi)=\sum_{x\in\mathbb F_q}\psi(x)\zeta_p^{{\mathop{\rm
Tr}\nolimits}(x)},
$$
where ${\mathop{\rm
Tr}\nolimits}(x)=x+x^p+x^{p^2}+\dots+x^{p^{s-1}}$ is the trace of
$x$ from $\mathbb F_q$ to $\mathbb F_p$. The next lemma gives an expression for $N[a_1^{}x_1^{d_1}+\dots+a_n^{}x_n^{d_n}=0]$ in terms of Gauss sums.
\begin{lemma}
\label{l1}
Let $a_1,\dots,a_n\in\mathbb F_q^*$, $d_1,\dots,d_n$ be positive integers, $d_j$ divides $q-1$ for all $j$, and let $\psi_j$ be a character of order $d_j$ on $\mathbb F_q$, $1\le j\le n$. Then
\begin{align*}
N&[a_1^{}x_1^{d_1}+\dots+a_n^{}x_n^{d_n}=0]\\
&=q^{n-1}+\frac{q-1}q\sum_{\substack{1\le j_1\le d_1-1\\ \dots\\1\le j_n\le d_n-1\\(j_1/d_1)+\dots+(j_n/d_n)\in\mathbb Z}}\bar\psi_1^{j_1}(a_1)\cdots\bar\psi_n^{j_n}(a_n) G(\psi_1^{j_1})\cdots G(\psi_n^{j_n}).
\end{align*}
\end{lemma}

\begin{proof}
See \cite[Theorems~10.3.1 and 10.4.2]{BEW} or \cite[Equation~(6.14)]{LN}.
\end{proof}

We recall some properties of Gauss sums, which will be used throughout this paper.
\begin{lemma}
\label{l2}
Let $\psi$ be a nontrivial character on $\mathbb F_q$. Then
\begin{itemize}
\item[\textup{(a)}] 
$G(\psi)G(\bar\psi)=\psi(-1)q$;
\item[\textup{(b)}] 
$G(\psi)=G(\psi^p)$.
\end{itemize}
\end{lemma}

\begin{proof}
See \cite[Theorem~1.1.4(a, d)]{BEW} or \cite[Theorem~5.12(iv,v)]{LN}.
\end{proof}

The evaluation of the quadratic Gauss sum $G(\eta)$ is given in the following lemma.
\begin{lemma}
\label{l3}
We have
$$
G(\eta)=
\begin{cases}
(-1)^{s-1}q^{1/2}&\text{if\,\, $p\equiv 1\!\!\pmod{4}$,}\\
(-1)^{s-1}i^s q^{1/2}&\text{if\,\, $p\equiv 3\!\!\pmod 4$.}
\end{cases}
$$
\end{lemma}

\begin{proof}
See \cite[Theorem~11.5.4]{BEW} or \cite[Theorem~5.15]{LN}.
\end{proof}

The next lemma is a particular case of the Stickelberger theorem.

\begin{lemma}
\label{l4}
Let $p\equiv 3\!\pmod{8}$, $2\mid s$ and $\psi$ be a biquadratic character on $\mathbb F_q$. Then $G(\psi)=-q^{1/2}$.
\end{lemma}

\begin{proof}
See \cite[Theorem~11.6.3]{BEW}.
\end{proof}

The following lemma is a special case of the Davenport-Hasse product formula for Gauss sums.
\begin{lemma}
\label{l5}
Let $\psi$ be a nontrivial character on $\mathbb F_q$
with $\psi\ne\eta$. Then
$$
G(\psi)G(\psi\eta)=\bar\psi(4)G(\psi^2)G(\eta).
$$
\end{lemma}

\begin{proof}
See \cite[Theorem~11.3.5]{BEW} or \cite[Corollary~5.29]{LN}.
\end{proof}

Let $\psi$ be a nontrivial character on $\mathbb F_q$. The Jacobi sum $J(\psi)$ over $\mathbb F_q$ is defined by
$$
J(\psi)=\sum_{x\in\mathbb F_q}\psi(x)\psi(1-x).
$$
An important relationship between Jacobi sums and Gauss sums is presented in the next lemma.

\begin{lemma}
\label{l6}
Let $\psi$ be a nontrivial character on $\mathbb F_q$
with $\psi\ne\eta$. Then
$$
G(\psi)^2=G(\psi^2)J(\psi).
$$
\end{lemma}

\begin{proof}
See \cite[Theorem~2.1.3(a)]{BEW} or \cite[Theorem~5.21]{LN}.
\end{proof}

Let $\psi$ be a character on $\mathbb F_q$. The lift $\psi'$ of
the character $\psi$ from $\mathbb F_{q^{\vphantom{r}}}$ to the
extension field $\mathbb F_{q^r}$ is given by
$$
\psi'(x)=\psi({\mathop{\rm N}}_{\mathbb F_{q^r}/\mathbb
F_{q^{\vphantom{r}}}}(x)), \qquad x\in\mathbb F_{q^r},
$$
where ${\mathop{\rm N}}_{\mathbb F_{q^r}/\mathbb
F_{q^{\vphantom{r}}}}(x)=x\cdot x^q\cdot x^{q^2}\cdots
x^{q^{r-1}}=x^{(q^r-1)/(q-1)}$ is the norm of $x$ from
$\mathbb F_{q^r}$ to $\mathbb F_{q^{\vphantom{r}}}$. The basic properties of the lift $\psi'$ of $\psi$ from $\mathbb F_{q^{\vphantom{r}}}$ to $\mathbb F_{q^r}$  are recorded in the next lemma.

\begin{lemma}
\label{l7} 
Let $\psi$ be a character on $\mathbb
F_{q^{\vphantom{r}}}$ and let $\psi'$ denote the lift of $\psi$
from $\mathbb F_{q^{\vphantom{r}}}$ to $\mathbb F_{q^r}$. Then
\begin{itemize}
\item[\textup{(a)}] 
$\psi'$ is a character on $\mathbb F_{q^r}$;
\item[\textup{(b)}]
a character $\lambda$ on $\mathbb F_{q^r}$ equals the lift $\psi'$ of some character $\psi$ on $\mathbb F_q$ if and only if the order of $\lambda$ divides $q-1$;
\item[\textup{(c)}]
$\psi'$ and $\psi$ have the same order.
\end{itemize}
\end{lemma}

\begin{proof}
See \cite[Theorem~11.4.4(a, c, e)]{BEW}.
\end{proof}

The following lemma, which is due to Davenport and Hasse, gives the relationship between a Gauss sum and its lift. 
\begin{lemma}
\label{l8} 
Let $\psi$ be a nontrivial character on $\mathbb F_q$
and let $\psi'$ denote the lift of $\psi$ from $\mathbb F_{q^{}}$
to $\mathbb F_{q^r}$. Then
$$
G(\psi')=(-1)^{r-1}G(\psi)^r.
$$
\end{lemma}

\begin{proof}
See \cite[Theorem~11.5.2]{BEW} or \cite[Theorem~5.14]{LN}.
\end{proof}

Now we turn to the case $p\equiv\pm 3\!\pmod{8}$. The next three lemmas were established in our earlier paper~\cite{B2} in more general settings (see Lemmas~2.2, 2.13, 2.16, respectively).

\begin{lemma}
\label{l9} 
Let $p\equiv\pm 3\!\pmod{8}$, $r$ be an
integer, and $\xi$ be a $2^k${\rm th} primitive root of unity,
where $r\ge 3$ and $k\le r$. Then
$$
\sum_{v=0}^{2^{r-2}-1}\xi^{p^v}=
\begin{cases}
2^{r-3}(\xi+\xi^p)&\text{if $k\le 3$,}\\
0&\text{if $k>3$.}
\end{cases}
$$
\end{lemma}

\begin{lemma}
\label{l10} 
Let $p\equiv\pm 3\!\pmod{8}$ and
$\psi$ be a character of order~$2^r$ on $\mathbb F_q$, where 
$$
r\ge\begin{cases}
4&\text{if $p\equiv\hphantom{-} 3\!\!\pmod{8}$,}\\
3&\text{if $p\equiv-3\!\!\pmod{8}$.}
\end{cases}
$$ 
Then $G(\psi)=G(\psi\eta)$.
\end{lemma}

\begin{lemma}
\label{l11} 
Let $p\equiv\pm 3\!\pmod{8}$  and
$\psi$ be a character of order~$2^r$ on $\mathbb F_q$, where $r\ge 3$. Then
$$
\psi(4)=
\begin{cases}
1 & \text{if $p\equiv\hphantom{-}3\!\!\pmod{8}$,}\\
(-1)^{s/2^{r-2}}&\text{if $p\equiv -3\!\!\pmod{8}$.}
\end{cases}
$$
\end{lemma}

We next relate Gauss sums over $\mathbb F_q$ to Jacobi sums over a subfield of $\mathbb F_q$.

\begin{lemma}
\label{l12}
Let $p\equiv 3\!\pmod{8}$ and $\psi$ be a character of order $2^r$  on $\mathbb F_q$, where $r\ge 3$. Assume that $2^{r+1}\mid(q-1)$. Then $\psi^{2^{r-3}}$  is equal to the lift of some octic character $\chi$ on $\mathbb F_{p^{s/2^{r-2}}}$. Moreover,  $G(\psi)=q^{(2^{r-2}-1)/2^{r-1}}J(\chi)$.
\end{lemma}

\begin{proof}
We prove the assertion of the lemma by induction on $r$. Let $16\mid(q-1)$ and $\psi$ be an octic character on $\mathbb F_q$. As $4\mid s$, we have $8\mid(p^{s/2}-1)$, and Lemma~\ref{l7} shows that $\psi$ is equal to the lift of some octic character $\chi$ on $\mathbb F_{p^{s/2}}$, that is, $\chi'=\psi$. Lemmas~\ref{l6} and \ref{l8} yield $G(\psi)=G(\chi')=-G(\chi)^2=-G(\chi^2)J(\chi)$. Note that $\chi^2$ has order 4. Thus, by Lemma~\ref{l4}, $G(\chi^2)=-q^{1/4}$,
and so $G(\psi)=q^{1/4}J(\chi)$. This completes the proof for the case $r=3$.

Suppose now that $r>3$, and assume that the result is true when $r$ is replaced by $r-1$. Let $2^{r+1}\mid(q-1)$ and $\psi$ be a character of order $2^r$  on $\mathbb F_q$. Since $s$ is even, we have $\nu_2(q-1)=\nu_2(p^s-1)=\nu_2(p^2-1)+\nu_2(s)-1$, where $\nu_2(z)$ denotes the $2$-adic valuation of $z\in\mathbb Z^+$, i.e., $2^{\nu_2(z)}\parallel z$ (for a proof, see \cite[Proposition~1]{B}). Hence $\nu_2(s)=\nu_2(q-1)-2\ge r-1$. Then $2^{r-2}\mid\frac s2$, and so $2^r\mid(p^{s/2}-1)$. By Lemma~\ref{l7}, $\psi$ is equal to the lift of some character $\rho$ of order $2^r$ on $\mathbb F_{p^{s/2}}$, that is $\rho'=\psi$. Applying Lemmas \ref{l3}, \ref{l5}, \ref{l8}, \ref{l10}, \ref{l11} and using the fact that $8\mid s$, we deduce
\begin{align}
G(\psi)&=G(\rho')=-G(\rho)^2=-G(\rho)G(\rho\eta_0)=-G(\rho^2)G(\eta_0)\notag\\
&=-(-1)^{(s/2)-1}i^{s/2}p^{s/4}G(\rho^2)=q^{1/4}G(\rho^2),\label{eq3}
\end{align}
where $\eta_0$ denotes the quadratic character on $\mathbb F_{p^{s/2}}$. Note that $\rho^2$ has order $2^{r-1}$ and $2^r\mid(p^{s/2}-1)$. Hence, by inductive hypothesis, $(\rho^2)^{2^{r-4}}=\rho^{2^{r-3}}$ is equal to the lift of some octic character $\chi$ on $\mathbb F_{p^{(s/2)/2^{r-3}}}=\mathbb F_{p^{s/2^{r-2}}}$ and $G(\rho^2)=(p^{s/2})^{(2^{r-3}-1)/2^{r-2}}J(\chi)=q^{(2^{r-3}-1)/2^{r-1}}J(\chi)$. Substituting this expression for $G(\rho^2)$ into \eqref{eq3}, we obtain $G(\psi)=q^{(2^{r-2}-1)/2^{r-1}}J(\chi)$. It remains to show that $\psi^{2^{r-3}}$ is equal to the lift of $\chi$. Indeed, for any $x\in\mathbb F_q$ we have
\begin{align*}
\chi({\mathop{\rm N}}_{\mathbb F_q/\mathbb F_{p^{s/2^{r-2}}}}(x))&=\chi(x^{(p^s-1)/(p^{s/2^{r-2}}-1)})
=\chi((x^{(p^s-1)/(p^{s/2}-1)})^{(p^{s/2}-1)/(p^{s/2^{r-2}}-1)})\\
&=\chi({\mathop{\rm N}}_{\mathbb F_{p^{s/2}}/\mathbb F_{p^{s/2^{r-2}}}}(x^{(p^s-1)/(p^{s/2}-1)}))=\rho^{2^{r-3}}(x^{(p^s-1)/(p^{s/2}-1)})\\
&=\left(\rho({\mathop{\rm N}}_{\mathbb F_{p^s}/\mathbb F_{p^{s/2}}}(x))\right)^{2^{r-3}}=\psi^{2^{r-3}}(x).
\end{align*}
Therefore $\chi'=\psi^{2^{r-3}}$, and the result now follows by the principle of mathematical induction.
\end{proof}

For the case $p\equiv -3\!\pmod{8}$ a similar result is given in the next lemma.
\begin{lemma}
\label{l13}
Let $p\equiv -3\!\pmod{8}$ and $\psi$ be a character of order $2^r$  on $\mathbb F_q$, where $r\ge 2$. Assume that $2^{r+1}\mid(q-1)$. Then $\psi^{2^{r-2}}$  is equal to the lift of some biquadratic character $\chi$ on $\mathbb F_{p^{s/2^{r-1}}}$. Moreover,  $G(\psi)=(-1)^{s(r-1)/2^{r-1}}q^{(2^{r-1}-1)/2^r}J(\chi)$.
\end{lemma}

\begin{proof}
The proof is analogous to that of Lemma~\ref{l12}.
\end{proof}

From now on we shall assume that $p\equiv\pm 3\!\pmod{8}$, $2^m\mid(q-1)$, $\lambda$ is a fixed character of order $2^m$ on $\mathbb F_q$ and
$$
m\ge\begin{cases}
3&\text{if $p\equiv\hphantom{-}3\!\!\pmod{8}$,}\\
2&\text{if $p\equiv-3\!\!\pmod{8}$.}
\end{cases}
$$
We observe that $2^{m-2}\mid s$. To simplify notation, put $N=N[x_1^{2^m}+\dots+x_n^{2^m}=0]$. Employing Lemma~\ref{l1}, we obtain
\begin{align}
N&=q^{n-1}+\frac{q-1}q\sum_{\substack{1\le j_1,\dots,j_n\le 2^m-1\\ j_1+\dots+j_n\equiv 0\!\!\!\!\pmod{2^m}}}G(\lambda^{j_1})\cdots G(\lambda^{j_n})\notag\\
&=q^{n-1}+\frac{q-1}{2^mq}\sum_{c=1}^{2^m}\left(\sum_{j=1}^{2^m-1}G(\lambda^j)\zeta_{2^m}^{cj}\right)^n.\label{eq4}
\end{align}
For $t=0,1,\dots,m$, set
$$
S_t=\sum_{\substack{c=1\\2^t\parallel c}}^{2^m}\left(\sum_{j=1}^{2^m-1} G(\lambda^j)\zeta_{2^m}^{cj}\right)^n=\sum_{\substack{c_0=1\\2\nmid c_0}}^{2^{m-t}}\left(\sum_{j=1}^{2^m-1} G(\lambda^j)\zeta_{2^{m-t}}^{c_0j}\right)^n.
$$
Then \eqref{eq4} can be rewritten in the form
\begin{equation}
\label{eq5}
N=q^{n-1}+\frac{q-1}{2^mq}\sum_{t=0}^m S_t.
\end{equation}
For $r=1,2,\dots,m$ and any odd integer $c_0$, set
$$
W_{r,t}(c_0)=\sum_{\substack{j=1\\ 2^{m-r}\parallel j}}^{2^m-1}G(\lambda^j)\zeta_{2^{m-t}}^{c_0j}=\sum_{\substack{j_0=1\\ 2\nmid j_0}}^{2^r-1}G(\lambda^{2^{m-r}j_0})\zeta_{2^{m-t}}^{2^{m-r}c_0j_0}.
$$
In this notation we can write
\begin{equation}
\label{eq6}
S_t=\sum_{\substack{c_0=1\\2\nmid c_0}}^{2^{m-t}}\left(\sum_{r=1}^m W_{r,t}(c_0)\right)^n.
\end{equation}

\begin{lemma}
\label{l14}
We have
\begin{align*}
W_{1,t}(c_0)&=\begin{cases}
-G(\eta)&\text{if $t=0$,}\\
\hphantom{-}G(\eta)&\text{if $t\ge 1$,}
\end{cases}\\
W_{2,t}(c_0)&=\begin{cases}
G(\lambda^{2^{m-2}})+G(\bar\lambda^{2^{m-2}})&\text{if $t\ge 2$,}\\
-\bigl(G(\lambda^{2^{m-2}})+G(\bar\lambda^{2^{m-2}})\bigr)&\text{if $t=1$,}\\
i^{c_0}\bigl(G(\lambda^{2^{m-2}})-G(\bar\lambda^{2^{m-2}})\bigr)&\text{if $t=0$,}
\end{cases}
\end{align*}
and, for $3\le r\le m$,
$$
W_{r,t}(c_0)=\begin{cases}
2^{r-2}\bigl(G(\lambda^{2^{m-r}})+G(\bar\lambda^{2^{m-r}})\bigr)&\text{if $r\le t$,}\\
-2^{r-2}\bigl(G(\lambda^{2^{m-r}})+G(\bar\lambda^{2^{m-r}})\bigr)&\text{if $r=t+1$,}\\
2^{r-2} i^{c_0}\bigl(G(\lambda^{2^{m-r}})-G(\bar\lambda^{2^{m-r}})\bigr)&\text{if $r=t+2$ and $p\equiv -3\!\!\pmod{8}$,}\\
2^{r-3} i\sqrt{2}\,\bigl(G(\lambda^{2^{m-r}})-G(\bar\lambda^{2^{m-r}})\bigr)&\text{if $r=t+3$, $p\equiv 3\!\!\pmod{8}$}\\
&\text{and $c_0\equiv 1$ or $3\!\!\pmod{8}$,}\\
-2^{r-3} i\sqrt{2}\,\bigl(G(\lambda^{2^{m-r}})-G(\bar\lambda^{2^{m-r}})\bigr)&\text{if $r=t+3$, $p\equiv 3\!\!\pmod{8}$}\\
&\text{and $c_0\equiv 5$ or $7\!\!\pmod{8}$,}\\
0&\text{otherwise.}
\end{cases}
$$
\end{lemma}

\begin{proof}
We first observe that
\begin{align*}
W_{1,t}(c_0)&=G(\lambda^{2^{m-1}})\zeta_{2^{m-t}}^{2^{m-1}c_0}=G(\eta)\zeta_{2^{m-t}}^{2^{m-1}c_0}=\begin{cases}
-G(\eta)&\text{if $t=0$,}\\
\hphantom{-}G(\eta)&\text{if $t\ge 1$,}
\end{cases}\\
W_{2,t}(c_0)&=G(\lambda^{2^{m-2}})\zeta_{2^{m-t}}^{2^{m-2}c_0}+G(\lambda^{3\cdot 2^{m-2}})\zeta_{2^{m-t}}^{3\cdot 2^{m-2}c_0}\\
&=\begin{cases}
G(\lambda^{2^{m-2}})+G(\bar\lambda^{2^{m-2}})&\text{if $t\ge 2$,}\\
-\bigl(G(\lambda^{2^{m-2}})+G(\bar\lambda^{2^{m-2}})\bigr)&\text{if $t=1$,}\\
i^{c_0}\bigl(G(\lambda^{2^{m-2}})-G(\bar\lambda^{2^{m-2}})\bigr)&\text{if $t=0$.}
\end{cases}
\end{align*}

Now assume that $3\le r\le m$. Since $\lambda^{2^{m-r}}$ has order $2^r$ and $\pm p^0,\pm p^1,\dots,\pm p^{2^{r-2}-1}$ is a reduced residue system modulo $2^r$, we conclude that
$$
W_{r,t}(c_0)=\sum_{u\in\{\pm 1\}}\sum_{v=0}^{2^{r-2}-1}G(\lambda^{2^{m-r}up^v})\zeta_{2^{m-t}}^{2^{m-r}c_0up^v}.
$$
Applying Lemma~\ref{l2}(b), we obtain
\begin{equation}
\label{eq7}
W_{r,t}(c_0)=G(\lambda^{2^{m-r}})\sum_{v=0}^{2^{r-2}-1}\zeta_{2^{m-t}}^{2^{m-r}c_0p^v}+G(\bar\lambda^{2^{m-r}})\sum_{v=0}^{2^{r-2}-1}\bar\zeta_{2^{m-t}}^{2^{m-r}c_0p^v}.
\end{equation}
If $r\le t$, then $\zeta_{2^{m-t}}^{2^{m-r}c_0}=\bar\zeta_{2^{m-t}}^{2^{m-r}c_0}=1$ and $W_{r,t}(c_0)=2^{r-2}\bigl(G(\lambda^{2^{m-r}})+G(\bar\lambda^{2^{m-r}})\bigr)$. Suppose that $r>t$. Then $\zeta_{2^{m-t}}^{2^{m-r}c_0}=\zeta_{2^{r-t}}^{c_0}$ is a $2^{r-t}$th primitive root of unity. If $r=t+1$, then $\zeta_{2^{r-t}}^{c_0}=-1$, and so $W_{r,t}(c_0)=-2^{r-2}\bigl(G(\lambda^{2^{m-r}})+G(\bar\lambda^{2^{m-r}})\bigr)$. If $r=t+2$, then $\zeta_{2^{r-t}}^{c_0}=i^{c_0}$. Appealing to Lemma~\ref{l9}, we deduce that
$$
\sum_{v=0}^{2^{r-2}-1}\zeta_{2^{m-t}}^{2^{m-r}c_0p^v}=\sum_{v=0}^{2^{r-2}-1}i^{c_0p^v}=2^{r-3}(i^{c_0}+i^{c_0p})=\begin{cases}
0&\text{if $p\equiv\hphantom{-}3\!\!\pmod{8}$,}\\
2^{r-2}i^{c_0}&\text{if $p\equiv -3\!\!\pmod{8}$,}
\end{cases}
$$
 and the result follows from \eqref{eq7}. Next assume that $r=t+3$. Again by Lemma~\ref{l9},
 \begin{align*}
 \sum_{v=0}^{2^{r-2}-1}\zeta_{2^{m-t}}^{2^{m-r}c_0p^v}&=\sum_{v=0}^{2^{r-2}-1}\zeta_8^{c_0p^v}=2^{r-3}(\zeta_8^{c_0}+\zeta_8^{c_0p})\\
 &=\begin{cases}
2^{r-3} i\sqrt{2}&\text{if $p\equiv\hphantom{-}3\!\!\pmod{8}$ and $c_0\equiv 1$ or $3\!\!\pmod{8}$,}\\
 -2^{r-3} i\sqrt{2}&\text{if $p\equiv\hphantom{-}3\!\!\pmod{8}$ and $c_0\equiv 5$ or $7\!\!\pmod{8}$,}\\
 0&\text{if $p\equiv -3\!\!\pmod{8}$.}\\
 \end{cases}
 \end{align*}
 The result now follows from \eqref{eq7} and the fact that $\bar\zeta_8^{c_0}+\bar\zeta_8^{3c_0}=-(\zeta_8^{c_0}+\zeta_8^{3c_0})$. Finally, assume that $r>t+3$. In view of Lemma~\ref{l9},
 $$
 \sum_{v=0}^{2^{r-2}-1}\zeta_{2^{m-t}}^{2^{m-r}c_0p^v}=\sum_{v=0}^{2^{r-2}-1}\zeta_{2^{r-t}}^{c_0p^v}=0,
 $$
 and \eqref{eq7} yields $W_{r,t}(c_0)=0$. This completes the proof of Lemma~\ref{l14}.
\end{proof}

From Lemma~\ref{l14} we see that $W_{r,m-1}(1)=W_{r,m}(1)$ for $1\le r\le m-1$, and $W_{m,m-1}(1)=-W_{m,m}(1)$.  Note also that in the case $p\equiv 3\!\pmod{8}$ we have $G(\lambda^{2^{m-2}})=G(\bar\lambda^{2^{m-2}})=-q^{1/2}$ by Lemma~\ref{l4}. Hence in this case $W_{2,0}(c_0)=0$ for any odd $c_0$. In view of these observations, the following corollary is an immediate consequence of Lemma~\ref{l14}.

\begin{corollary}
\label{c1}
We have
$$
S_{m-1}+S_m=\left(\sum_{r=1}^{m-1} W_{r,m}(1)+W_{m,m}(1)\right)^n+\left(\sum_{r=1}^{m-1} W_{r,m}(1)-W_{m,m}(1)\right)^n.
$$
Furthermore, if $p\equiv 3\!\pmod{8}$, then
$$
S_{m-2}=2\cdot\left(\sum_{r=1}^{m-1} W_{r,m-2}(1)\right)^n,
$$
and, for $t\le m-3$,
$$
S_t=2^{m-t-2}\left[\left(\sum_{r=1}^{t+1} W_{r,t}(1)+W_{t+3,t}(1)\right)^n+\left(\sum_{r=1}^{t+1} W_{r,t}(1)-W_{t+3,t}(1)\right)^n\,\right].
$$
If $p\equiv -3\!\pmod{8}$ and $t\le m-2$, then
$$
S_t=2^{m-t-2}\left[\left(\sum_{r=1}^{t+1} W_{r,t}(1)+W_{t+2,t}(1)\right)^n+\left(\sum_{r=1}^{t+1} W_{r,t}(1)-W_{t+2,t}(1)\right)^n\,\right].
$$
\end{corollary}

\section{The case $p\equiv 3\pmod{8}$}
\label{s3}

In this section, let $p\equiv 3\!\pmod{8}$, $q=p^s\equiv 1\!\pmod{2^m}$, $m\ge 3$. As before, $\lambda$ is a fixed character of order $2^m$ on $\mathbb F_q$.

For $r=2,3,\dots,m$, define  the integers $A_r$ and $B_r$ by
\begin{equation}
\label{eq8}
p^{s/2^{r-2}}=A_r^2+2B_r^2,\qquad A_r\equiv -1\!\!\pmod{4},\qquad p\nmid A_r.
\end{equation}
It is well known \cite[Lemma~3.0.1]{BEW} that for each fixed $r$, \eqref{eq8} determines $A_r$ uniquely but determines $B_r$ only up to sign. Also, if $2^{r-1}\mid s$, or, equivalently, $2^{r+1}\mid(q-1)$, and $\chi$ is an octic character on $\mathbb F_{p^{s/2^{r-2}}}$ then $J(\chi)=A_r\pm|B_r| i\sqrt{2}$  (see \cite[Lemma~17]{B1}). Combining this last fact with Lemma~\ref{l12}, we deduce the following result.

\begin{lemma}
\label{l15}
Let $r$ be an integer with $3\le r\le m$ and assume that $2^{r+1}\mid(q-1)$. Then
$$
G(\lambda^{2^{m-r}})+G(\bar\lambda^{2^{m-r}})=2A_r q^{(2^{r-2}-1)/2^{r-1}}
$$
and
$$
G(\lambda^{2^{m-r}})-G(\bar\lambda^{2^{m-r}})=\pm 2|B_r| q^{(2^{r-2}-1)/2^{r-1}}i\sqrt{2}.
$$
\end{lemma}

Lemma~\ref{l15} allows us to evaluate $G(\lambda^{2^{m-r}})+G(\bar\lambda^{2^{m-r}})$ and $G(\lambda^{2^{m-r}})-G(\bar\lambda^{2^{m-r}})$ (in the latter case only up to sign) if either $3\le r\le m-1$ or $r=m$ and $2^{m+1}\mid(q-1)$. For the remaining case $r=m$ and $2^m\parallel(q-1)$, we need the following lemma.

\begin{lemma}
\label{l16}
Assume that $2^m\parallel(q-1)$. Then
$$
G(\lambda)+G(\bar\lambda)=\pm 2A_m q^{(2^{m-2}-1)/2^{m-1}}i
$$
and
$$
G(\lambda)-G(\bar\lambda)=\pm 2|B_m| q^{(2^{m-2}-1)/2^{m-1}}\sqrt{2}.
$$
\end{lemma}

\begin{proof}
Since $2^m\parallel(q-1)$, it follows from Lemma~\ref{l2}(a) that
$$
\bigl(G(\lambda)+G(\bar\lambda)\bigr)^2=G(\lambda)^2+G(\bar\lambda)^2+ 2\lambda(-1)q=G(\lambda)^2+G(\bar\lambda)^2-2q.
$$
If $m=3$, then, by Lemmas~\ref{l4} and \ref{l6},
$$
G(\lambda)^2+G(\bar\lambda)^2=G(\lambda^2)J(\lambda)+G(\bar\lambda^2)J(\bar\lambda)=-2A_2 q^{1/2}.
$$
If $m\ge 4$, then Lemmas~\ref{l3}, \ref{l5}, \ref{l10}, \ref{l11} and \ref{l15} yield
\begin{align*}
G(\lambda)^2+G(\bar\lambda)^2&=G(\lambda)G(\lambda\eta)+G(\bar\lambda)G(\bar\lambda\eta)=\bar\lambda(4)G(\lambda^2)G(\eta)+\lambda(4)G(\bar\lambda^2)G(\eta)\\
&=-q^{1/2}(G(\lambda^2)+G(\bar\lambda^2))=-2A_{m-1}q^{(2^{m-2}-1)/2^{m-2}}.
\end{align*}
Thus, in both cases,
\begin{equation}
\label{eq9}
\bigl(G(\lambda)+ G(\bar\lambda)\bigr)^2=-2q^{(2^{m-2}-1)/2^{m-2}}(A_{m-1}+p^{s/2^{m-2}}).
\end{equation}
 Note that
$$
A_{m-1}^2+2B_{m-1}^2=p^{s/2^{m-3}}=(p^{s/2^{m-2}})^2=(A_m^2+2B_m^2)^2
=(A_m^2-2B_m^2)^2+2\cdot(2A_mB_m)^2.
$$
Hence $A_{m-1}=\pm(A_m^2-2B_m^2)$. Since $p^{s/2^{m-2}}=A_m^2+2B_m^2\equiv 3\!\pmod{8}$, $B_m$ is odd, and so $A_{m-1}=A_m^2-2B_m^2$. Substituting the expressions for $p^{s/2^{m-2}}$ and $A_{m-1}$ into \eqref{eq9} and taking square roots of both sides, we find that $G(\lambda)+ G(\bar\lambda)=\pm 2A_mq^{(2^{m-2}-1)/2^{m-1}}i$. Similarly,
$$
\bigl(G(\lambda)-G(\bar\lambda)\bigr)^2=-2q^{(2^{m-2}-1)/2^{m-2}}(A_{m-1}-p^{s/2^{m-2}})=8B_m^2q^{(2^{m-2}-1)/2^{m-2}},
$$
which implies that $G(\lambda)-G(\bar\lambda)=\pm 2|B_m| q^{(2^{m-2}-1)/2^{m-1}}\sqrt{2}$.
\end{proof}

We are now ready to determine the number $N$  of solutions to \eqref{eq2} in the case $p\equiv 3\!\pmod{8}$. In the proofs of the next two theorems, we shall frequently employ Lemmas~\ref{l14}--\ref{l16} and Corollary~\ref{c1} without further comments.

\begin{theorem}
\label{t1} 
Let $p\equiv 3\!\pmod{8}$ and $2^{m+1}\mid (q-1)$. If $m=3$ then
\begin{align*}
N=\,&q^{n-1}+\frac{q-1}{8q}\biggl[\biggr.2\cdot\left((q^{\frac12}+4B_3q^{\frac14})^n+(q^{\frac12}-4B_3q^{\frac14})^n\right)+2q^{\frac n2}\\
&\biggl.+\left(-3q^{\frac12}+4A_3q^{\frac14}\right)^n+\left(-3q^{\frac12}-4A_3q^{\frac14}\right)^n\biggr].
\end{align*}
If $m\ge 4$ then
\begin{align*}
N=\,&q^{n-1}+\frac{q-1}{2^mq}\cdot\biggl[2^{m-2}\cdot\Bigl((q^{\frac12}+4B_3q^{\frac14})^n+(q^{\frac12}-4B_3q^{\frac14})^n\Bigr)\biggr.\\
&+2^{m-3}\cdot\Bigl((q^{\frac12}+8B_4q^{\frac38})^n+(q^{\frac12}-8B_4q^{\frac38})^n\Bigr)\\
&\begin{aligned}
+\sum_{t=2}^{m-3}\!2^{m-t-2}\Bigl(&(-3q^{\frac12}+\sum_{r=3}^t 2^{r-1}A_rq^{\frac{2^{r-2}-1}{2^{r-1}}}-2^t A_{t+1}q^{\frac{2^{t-1}-1}{2^t}}+2^{t+2} B_{t+3}q^{\frac{2^{t+1}-1}{2^{t+2}}})^n\Bigr.\\
+&(-3q^{\frac12}+\sum_{r=3}^t 2^{r-1}A_rq^{\frac{2^{r-2}-1}{2^{r-1}}}-2^t A_{t+1}q^{\frac{2^{t-1}-1}{2^t}}-2^{t+2} B_{t+3}q^{\frac{2^{t+1}-1}{2^{t+2}}})^n\Bigl.\Bigr)\end{aligned}\\
&+2\cdot\Bigl(-3q^{\frac 12}+\sum_{r=3}^{m-2} 2^{r-1}A_rq^{\frac{2^{r-2}-1}{2^{r-1}}}-2^{m-2} A_{m-1}q^{\frac{2^{m-3}-1}{2^{m-2}}}\Bigr)^n\\
&+\Bigl(-3q^{\frac 12}+\sum_{r=3}^{m-1}2^{r-1}A_r q^{\frac{2^{r-2}-1}{2^{r-1}}}-2^{m-1}A_m q^{\frac{2^{m-2}-1}{2^{m-1}}}\Bigr)^n\\
&+\biggl.\Bigl(-3q^{\frac 12}+\sum_{r=3}^{m}2^{r-1}A_r q^{\frac{2^{r-2}-1}{2^{r-1}}}\Bigr)^n\biggr].
\end{align*}
The integers $A_r$ and $|B_r|$ are uniquely determined by~\eqref{eq8}.
\end{theorem}

\begin{proof}
Since $2^{m+1}\mid(q-1)$, $m\ge 3$ and $\lambda^{2^{m-2}}$ has order four, we see that 
\begin{align*}
W_{1,m}(1)&=-q^{1/2},\qquad W_{2,m}(1)=-2q^{1/2},\\ 
W_{r,m}(1)&=2^{r-1}A_rq^{(2^{r-2}-1)/2^{r-1}},\quad 3\le r\le m. 
\end{align*}
Hence
\begin{align}
S_{m-1}+S_m=\,&\Bigl(-3q^{\frac 12}+\sum_{r=3}^{m-1}2^{r-1}A_r q^{\frac{2^{r-2}-1}{2^{r-1}}}-2^{m-1}A_m q^{\frac{2^{m-2}-1}{2^{m-1}}}\Bigr)^n\notag\\
&+\Bigl(-3q^{\frac 12}+\sum_{r=3}^{m}2^{r-1}A_r q^{\frac{2^{r-2}-1}{2^{r-1}}}\Bigr)^n.\label{eq10}
\end{align}
Next, 
\begin{align*}
W_{1,m-2}(1)&=-q^{1/2},\qquad
W_{2,m-2}(1)=\begin{cases}
2q^{1/2}&\text{if $m=3$,}\\
-2q^{1/2}&\text{if $m\ge 4$,}
\end{cases}\\
W_{r,m-2}(1)&=2^{r-1}A_rq^{(2^{r-2}-1)/2^{r-1}},\quad 3\le r\le m-2,\quad m\ge 5,\\
W_{m-1,m-2}(1)&=-2^{m-2}A_{m-1}q^{(2^{m-3}-1)/2^{m-2}},\quad m\ge 4.
\end{align*}
Thus
\begin{equation}
\label{eq11}
S_{m-2}=\begin{cases}
2q^{\frac n2}&\text{if $m=3$,}\\
2\cdot\Bigl(-3q^{\frac 12}+\sum\limits_{r=3}^{m-2} 2^{r-1}A_rq^{\frac{2^{r-2}-1}{2^{r-1}}}-2^{m-2} A_{m-1}q^{\frac{2^{m-3}-1}{2^{m-2}}}\Bigr)^n&\text{if $m\ge 4$.}
\end{cases}
\end{equation}
Now assume that $2\le t\le m-3$, $m\ge 5$. Then
\begin{align*}
W_{1,t}(1)&=-q^{1/2},\qquad W_{2,t}(1)=-2q^{1/2},\\
W_{r,t}(1)&=2^{r-1}A_rq^{(2^{r-2}-1)/2^{r-1}},\quad 3\le r\le t,\quad t\ge 3,\\
W_{t+1,t}(1)&=-2^t A_{t+1}q^{(2^{t-1}-1)/2^t},\qquad W_{t+3,t}(1)=\pm 2^{t+2}|B_{t+3}|q^{(2^{t+1}-1)/2^{t+2}}.
\end{align*}
Therefore,
\begin{align}
S_t=\,&2^{m-t-2}\Bigl((-3q^{\frac12}+\sum_{r=3}^t 2^{r-1}A_rq^{\frac{2^{r-2}-1}{2^{r-1}}}-2^t A_{t+1}q^{\frac{2^{t-1}-1}{2^t}}+2^{t+2} B_{t+3}q^{\frac{2^{t+1}-1}{2^{t+2}}})^n\Bigr.\notag\\
&+(-3q^{\frac12}+\sum_{r=3}^t 2^{r-1}A_rq^{\frac{2^{r-2}-1}{2^{r-1}}}-2^t A_{t+1}q^{\frac{2^{t-1}-1}{2^t}}-2^{t+2} B_{t+3}q^{\frac{2^{t+1}-1}{2^{t+2}}})^n\Bigl.\Bigr).\label{eq12}
\end{align}
If $m\ge 4$, then
$$
W_{1,1}(1)=-q^{1/2},\qquad W_{2,1}(1)=2q^{1/2},\qquad W_{4,1}(1)=\pm 8|B_4|q^{3/8}.
$$
This yields
\begin{equation}
\label{eq13}
S_1=2^{m-3}\cdot\Bigl((q^{\frac12}+8B_4q^{\frac38})^n+(q^{\frac12}-8B_4q^{\frac38})^n\Bigr).
\end{equation}
Finally, we have
$$
W_{1,0}(1)=q^{1/2},\qquad W_{3,0}(1)=\pm 4|B_3| q^{1/4},
$$
and so
\begin{equation}
\label{eq14}
S_0=2^{m-2}\cdot\Bigl((q^{\frac12}+4B_3q^{\frac14})^n+(q^{\frac12}-4B_3q^{\frac14})^n\Bigr).
\end{equation}
Substituting \eqref{eq10}--\eqref{eq14} into \eqref{eq5}, we obtain the asserted result.
\end{proof}

\begin{theorem}
\label{t2}
 Let $p\equiv 3\!\pmod{8}$ and $2^m\parallel (q-1)$.  If $m=3$ then
\begin{align*}
N=\,&q^{n-1}+\frac{q-1}{8q}\biggl[2\cdot\left((-q^{\frac12}+4B_3q^{\frac14}i)^n+(-q^{\frac12}-4B_3q^{\frac14}i)^n\right)+2\cdot 3^nq^{\frac n2}\biggr.\\
&+\left(-q^{\frac12}+4A_3q^{\frac14}i\right)^n+\left(-q^{\frac12}-4A_3q^{\frac14}i\right)^n\biggl.\biggr].
\end{align*}
 If $m=4$ then
\begin{align*}
N=\,&q^{n-1}+\frac{q-1}{16q}\biggl[4\cdot\left((q^{\frac12}+4B_3q^{\frac14})^n+(q^{\frac12}-4B_3q^{\frac14})^n\right)\biggr.\\
&+2\cdot\left((q^{\frac12}+8B_4q^{\frac38}i)^n+(q^{\frac12}-8B_4q^{\frac38}i)^n\right)+2\cdot\left(-3q^{\frac12}-4A_3q^{\frac14}\right)^n\\
&+\left(-3q^{\frac12}+4A_3q^{\frac14}+8A_4q^{\frac38}i\right)^n+\left(-3q^{\frac12}+4A_3q^{\frac14}-8A_4q^{\frac38}i\right)^n\biggl.\biggr].
\end{align*}
If $m\ge 5$  then
\begin{align*}
N=\,&q^{n-1}+\frac{q-1}{2^mq}\biggl[2^{m-2}\cdot\Bigl((q^{\frac12}+4B_3q^{\frac14})^n+(q^{\frac12}-4B_3q^{\frac14})^n\Bigr)\biggr.\\
&+2^{m-3}\cdot\Bigl((q^{\frac12}+8B_4q^{\frac38})^n+(q^{\frac12}-8B_4q^{\frac38})^n\Bigr)\\
&\begin{aligned}
+\sum_{t=2}^{m-4}\!2^{m-t-2}\Bigl(&(-3q^{\frac12}+\sum_{r=3}^t 2^{r-1}A_rq^{\frac{2^{r-2}-1}{2^{r-1}}}-2^t A_{t+1}q^{\frac{2^{t-1}-1}{2^t}}+2^{t+2} B_{t+3}q^{\frac{2^{t+1}-1}{2^{t+2}}})^n\Bigr.\\
+&(-3q^{\frac12}+\sum_{r=3}^t 2^{r-1}A_rq^{\frac{2^{r-2}-1}{2^{r-1}}}-2^t A_{t+1}q^{\frac{2^{t-1}-1}{2^t}}-2^{t+2} B_{t+3}q^{\frac{2^{t+1}-1}{2^{t+2}}})^n\Bigl.\Bigr)\end{aligned}\\
&\begin{aligned}
+2\cdot\Bigl(&(-3q^{\frac 12}+\sum_{r=3}^{m-3} 2^{r-1}A_rq^{\frac{2^{r-2}-1}{2^{r-1}}}-2^{m-3} A_{m-2}q^{\frac{2^{m-4}-1}{2^{m-3}}}+2^{m-1}B_mq^{\frac{2^{m-2}-1}{2^{m-1}}}i)^n\Bigr.\\
+&(-3q^{\frac 12}+\sum_{r=3}^{m-3} 2^{r-1}A_rq^{\frac{2^{r-2}-1}{2^{r-1}}}-2^{m-3} A_{m-2}q^{\frac{2^{m-4}-1}{2^{m-3}}}-2^{m-1}B_mq^{\frac{2^{m-2}-1}{2^{m-1}}}i)^n\Bigl.\Bigr)\end{aligned}\\
&+2\cdot\Bigl(-3q^{\frac 12}+\sum_{r=3}^{m-2} 2^{r-1}A_rq^{\frac{2^{r-2}-1}{2^{r-1}}}-2^{m-2} A_{m-1}q^{\frac{2^{m-3}-1}{2^{m-2}}}\Bigr)^n\\
&+\Bigl(-3q^{\frac 12}+\sum_{r=3}^{m-1}2^{r-1}A_r q^{\frac{2^{r-2}-1}{2^{r-1}}}+2^{m-1}A_m q^{\frac{2^{m-2}-1}{2^{m-1}}}i\Bigr)^n\\
&+\biggl.\Bigl(-3q^{\frac 12}+\sum_{r=3}^{m-1}2^{r-1}A_r q^{\frac{2^{r-2}-1}{2^{r-1}}}-2^{m-1}A_m q^{\frac{2^{m-2}-1}{2^{m-1}}}i\Bigr)^n\biggr].
\end{align*}
The integers $A_r$ and $|B_r|$ are uniquely determined by~\eqref{eq8}.
\end{theorem}

\begin{proof}
Since $2^m\parallel(q-1)$, we find that
\begin{align*}
W_{1,m}(1)&=\begin{cases}
q^{1/2}&\text{if $m=3$,}\\
-q^{1/2}&\text{if $m\ge 4$,}
\end{cases}\qquad W_{2,m}(1)=-2q^{1/2},\\
W_{r,m}(1)&=2^{r-1}A_rq^{(2^{r-2}-1)/2^{r-1}},\quad 3\le r\le m-1,\quad m\ge 4,\\
W_{m,m}(1)&=\pm 2^{m-1}A_m q^{(2^{m-2}-1)/2^{m-1}}i.
\end{align*}
This yields
\begin{equation}
\label{eq15}
S_{m-1}+S_m=\left(-q^{\frac12}+4A_3q^{\frac14}i\right)^n+\left(-q^{\frac12}-4A_3q^{\frac14}i\right)^n
\end{equation}
if $m=3$, and
\begin{align}
S_{m-1}+S_m=\,&\Bigl(-3q^{\frac 12}+\sum_{r=3}^{m-1}2^{r-1}A_r q^{\frac{2^{r-2}-1}{2^{r-1}}}+2^{m-1}A_m q^{\frac{2^{m-2}-1}{2^{m-1}}}i\Bigr)^n\notag\\
&+\biggl.\Bigl(-3q^{\frac 12}+\sum_{r=3}^{m-1}2^{r-1}A_r q^{\frac{2^{r-2}-1}{2^{r-1}}}-2^{m-1}A_m q^{\frac{2^{m-2}-1}{2^{m-1}}}i\Bigr)^n\label{eq16}
\end{align}
if $m\ge 4$. Furthermore,
\begin{align*}
W_{1,m-2}(1)&=\begin{cases}
q^{1/2}&\text{if $m=3$,}\\
-q^{1/2}&\text{if $m\ge 4$,}
\end{cases}\qquad
W_{2,m-2}(1)=\begin{cases}
2q^{1/2}&\text{if $m=3$,}\\
-2q^{1/2}&\text{if $m\ge 4$,}
\end{cases}\\
W_{r,m-2}(1)&=2^{r-1}A_rq^{(2^{r-2}-1)/2^{r-1}},\quad 3\le r\le m-2,\quad m\ge 5,\\
W_{m-1,m-2}(1)&=-2^{m-2}A_{m-1}q^{(2^{m-3}-1)/2^{m-2}},\quad m\ge 4.
\end{align*}
Hence
\begin{equation}
\label{eq17}
S_{m-2}=\begin{cases}
2\cdot 3^nq^{\frac n2}&\text{if $m=3$,}\\
2\cdot\Bigl(-3q^{\frac 12}+\sum\limits_{r=3}^{m-2} 2^{r-1}A_rq^{\frac{2^{r-2}-1}{2^{r-1}}}-2^{m-2} A_{m-1}q^{\frac{2^{m-3}-1}{2^{m-2}}}\Bigr)^n&\text{if $m\ge 4$.}
\end{cases}
\end{equation}
If $m\ge 4$, then
\begin{align*}
W_{1,m-3}(1)&=-q^{1/2},\qquad W_{2,m-3}(1)=\begin{cases}
2q^{1/2}&\text{if $m=4$,}\\
-2q^{1/2}&\text{if $m\ge 5$,}
\end{cases}\\
W_{r,m-3}(1)&=2^{r-1}A_rq^{(2^{r-2}-1)/2^{r-1}},\quad 3\le r\le m-3,\quad m\ge 6,\\
W_{m-2,m-3}(1)&=-2^{m-3}A_{m-2}q^{(2^{m-4}-1)/2^{m-3}},\quad m\ge 5,\\
W_{m,m-3}(1)&=\pm 2^{m-1}|B_m|q^{(2^{m-2}-1)/2^{m-1}}i.
\end{align*}
Therefore,
\begin{equation}
\label{eq18}
S_{m-3}=2\cdot\left((q^{\frac12}+8B_4q^{\frac38}i)^n+(q^{\frac12}-8B_4q^{\frac38}i)^n\right)
\end{equation}
if $m=4$, and
\begin{align}
S_{m-3}=\,&2\cdot\Bigl((-3q^{\frac 12}+\sum_{r=3}^{m-3} 2^{r-1}A_rq^{\frac{2^{r-2}-1}{2^{r-1}}}-2^{m-3} A_{m-2}q^{\frac{2^{m-4}-1}{2^{m-3}}}+2^{m-1}B_mq^{\frac{2^{m-2}-1}{2^{m-1}}}i)^n\Bigr.\notag\\
&+(-3q^{\frac 12}+\sum_{r=3}^{m-3} 2^{r-1}A_rq^{\frac{2^{r-2}-1}{2^{r-1}}}-2^{m-3} A_{m-2}q^{\frac{2^{m-4}-1}{2^{m-3}}}-2^{m-1}B_mq^{\frac{2^{m-2}-1}{2^{m-1}}}i)^n\Bigl.\Bigr)\label{eq19}
\end{align}
if $m\ge 5$. It is easy to check that $S_2,\dots,S_{m-4}$ (for $m\ge 6$) and $S_1$ (for $m\ge 5$) are determined by \eqref{eq12} and \eqref{eq13}, respectively. Moreover, if $m\ge 4$, then $S_0$ is determined by \eqref{eq14}. For $m=3$, we have
$$
W_{1,0}(1)=-q^{1/2},\qquad W_{3,0}(1)=\pm 4|B_3|q^{1/4}i,
$$
and so
\begin{equation}
\label{eq20}
S_0=2\cdot\left((-q^{\frac12}+4B_3q^{\frac14}i)^n+(-q^{\frac12}-4B_3q^{\frac14}i)^n\right).
\end{equation}
Substituting \eqref{eq12}--\eqref{eq20} into \eqref{eq5}, we obtain the desired result.
\end{proof}

\section{The case $p\equiv -3\pmod{8}$}
\label{s4}

In this section, let $p\equiv -3\!\pmod{8}$, $q=p^s\equiv 1\!\pmod{2^m}$, $m\ge 2$. As in the previous sections, $\lambda$ denotes a fixed character of order $2^m$ on $\mathbb F_q$.

For $r=1,2,\dots,m-1$, define  the integers $C_r$ and $D_r$ by
\begin{equation}
\label{eq21}
p^{s/2^{r-1}}=C_r^2+D_r^2,\qquad C_r\equiv -1\!\!\pmod{4},\qquad p\nmid C_r.
\end{equation}
If $2^{m+1}\mid(q-1)$ (or, equivalently, $2^{m-1}\mid s$), we extend this notation to $r=m$. It is well known \cite[Lemma~3.0.1]{BEW} that for each fixed $r$, \eqref{eq21} determines $C_r$ uniquely but determines $D_r$ only up to sign. Further, if $\chi$ is a biquadratic character on $\mathbb F_{p^{s/2^{r-1}}}$ then $J(\chi)=C_r\pm|D_r| i$  (see \cite[Proposition~2]{KR}). Appealing to Lemma~\ref{l13}, we obtain the following result.

\begin{lemma}
\label{l17}
Let $r$ be an integer with $2^{r+1}\mid(q-1)$ and $2\le r\le m$. Then
$$
G(\lambda^{2^{m-r}})+G(\bar\lambda^{2^{m-r}})=
\begin{cases}
2C_r q^{(2^{r-1}-1)/2^r} &\text{if $2^{r+2}\mid(q-1)$,}\\
(-1)^{r-1}\cdot 2C_r q^{(2^{r-1}-1)/2^r} &\text{if $2^{r+1}\parallel(q-1)$,}
\end{cases}
$$
and
$$
G(\lambda^{2^{m-r}})-G(\bar\lambda^{2^{m-r}})=
\pm 2|D_r| q^{(2^{r-1}-1)/2^r}i.
$$
\end{lemma}
\medskip

To find $G(\lambda)\pm G(\bar\lambda)$ in the case when $2^m\parallel(q-1)$, we need the next result.

\begin{lemma}
\label{l18}
Assume that $2^m\parallel(q-1)$. Then
$$
G(\lambda)+G(\bar\lambda)=\pm q^{(2^{m-1}-1)/2^m}i\sqrt{2(q^{1/2^{m-1}}-(-1)^m C_{m-1})}
$$
and
$$
G(\lambda)-G(\bar\lambda)=\pm q^{(2^{m-1}-1)/2^m}\sqrt{2(q^{1/2^{m-1}}+(-1)^m C_{m-1})}\, .
$$
\end{lemma}

\begin{proof}
By employing the same type of argument as in the proof of Lemma~\ref{l16}, we see that
\begin{align*}
\bigl(G(\lambda)+G(\bar\lambda)\bigr)^2&=-2q^{(2^{m-1}-1)/2^{m-1}}\bigl(q^{1/2^{m-1}}-(-1)^m C_{m-1}\bigr),\\
\bigl(G(\lambda)-G(\bar\lambda)\bigr)^2&=\hphantom{-}2q^{(2^{m-1}-1)/2^{m-1}}\bigl(q^{1/2^{m-1}}+(-1)^m C_{m-1}\bigr).
\end{align*}
As $q^{1/2^{m-2}}=p^{s/2^{m-2}}=C_{m-1}^2+D_{m-1}^2$, we have $q^{1/2^{m-1}}>|C_{m-1}|$, and the result follows.
\end{proof}

We are now in a position to derive explicit formulas for $N$ when $p\equiv -3\!\pmod{8}$.  We shall be using Lemmas~\ref{l14}, \ref{l17}, \ref{l18} and Corollary~\ref{c1} without mention.

\begin{theorem}
\label{t3} 
Let $p\equiv -3\!\pmod{8}$ and $2^{m+1}\mid (q-1)$. Then
\begin{align*}
N=\,&q^{n-1}+\frac{q-1}{2^mq}\cdot\biggl[2^{m-2}\cdot\Bigl((q^{\frac12}+2D_2q^{\frac14})^n+(q^{\frac12}-2D_2q^{\frac14})^n\Bigr)\biggr.\\
&\begin{aligned}
+\sum_{t=1}^{m-2}2^{m-t-2}\Bigl(&(-q^{\frac12}+\sum_{r=2}^t 2^{r-1}C_rq^{\frac{2^{r-1}-1}{2^{r}}}-2^t C_{t+1}q^{\frac{2^{t}-1}{2^{t+1}}}+2^{t+1} D_{t+2}q^{\frac{2^{t+1}-1}{2^{t+2}}})^n\Bigr.\\
+&(-q^{\frac12}+\sum_{r=2}^t 2^{r-1}C_rq^{\frac{2^{r-1}-1}{2^{r}}}-2^t C_{t+1}q^{\frac{2^{t}-1}{2^{t+1}}}-2^{t+1} D_{t+2}q^{\frac{2^{t+1}-1}{2^{t+2}}})^n\Bigl.\Bigr)\end{aligned}\\
&+\Bigl(-q^{\frac 12}+\sum_{r=2}^{m-1}2^{r-1}C_r q^{\frac{2^{r-1}-1}{2^{r}}}-2^{m-1}C_m q^{\frac{2^{m-1}-1}{2^{m}}}\Bigr)^n\\
&+\biggl.\Bigl(-q^{\frac 12}+\sum_{r=2}^{m}2^{r-1}C_r q^{\frac{2^{r-1}-1}{2^{r}}}\Bigr)^n\biggr].
\end{align*}
The integers $C_r$ and $|D_r|$ are uniquely determined by~\eqref{eq21}.
\end{theorem}

\begin{proof}
We have
\begin{align*}
W_{1,m}(1)&=-q^{1/2},\\
W_{r,m}(1)&=2^{r-1}C_r q^{(2^{r-1}-1)/2^r},\quad 2\le r\le m-1,\quad m\ge 3,\\
W_{m,m}(1)&=\pm 2^{m-1}C_m q^{(2^{m-1}-1)/2^m}.
\end{align*}
Consequently,
\begin{align}
S_{m-1}+S_m=\,&\Bigl(-q^{\frac 12}+\sum_{r=2}^{m-1}2^{r-1}C_r q^{\frac{2^{r-1}-1}{2^{r}}}-2^{m-1}C_m q^{\frac{2^{m-1}-1}{2^{m}}}\Bigr)^n\notag\\
&+\biggl.\Bigl(-q^{\frac 12}+\sum_{r=2}^{m}2^{r-1}C_r q^{\frac{2^{r-1}-1}{2^{r}}}\Bigr)^n.\label{eq22}
\end{align}
Assume now that $1\le t\le m-2$, $m\ge 3$. Then
\begin{align*}
W_{1,t}(1)&=-q^{1/2},\\
W_{r,t}(1)&=2^{r-1}C_r q^{(2^{r-1}-1)/2^r},\quad 2\le r\le t,\quad t\ge 2,\\
W_{t+1,t}(1)&=-2^t C_{t+1}q^{(2^t-1)/2^{t+1}},\qquad W_{t+2,t}(1)=\pm 2^{t+1}|D_{t+2}|q^{(2^{t+1}-1)/2^{t+2}}.
\end{align*}
Thus
\begin{align}
S_t=\,&2^{m-t-2}\Bigl((-q^{\frac12}+\sum_{r=2}^t 2^{r-1}C_rq^{\frac{2^{r-1}-1}{2^{r}}}-2^t C_{t+1}q^{\frac{2^{t}-1}{2^{t+1}}}+2^{t+1} D_{t+2}q^{\frac{2^{t+1}-1}{2^{t+2}}})^n\Bigr.\notag\\
+&(-q^{\frac12}+\sum_{r=2}^t 2^{r-1}C_rq^{\frac{2^{r-1}-1}{2^{r}}}-2^t C_{t+1}q^{\frac{2^{t}-1}{2^{t+1}}}-2^{t+1} D_{t+2}q^{\frac{2^{t+1}-1}{2^{t+2}}})^n\Bigl.\Bigr).\label{eq23}
\end{align}
Finally,
$$
W_{1,0}(1)=q^{1/2},\qquad W_{2,0}(1)=\pm 2|D_2|q^{1/4},
$$
and so
\begin{equation}
\label{eq24}
S_0=2^{m-2}\cdot\Bigl((q^{\frac12}+2D_2q^{\frac14})^n+(q^{\frac12}-2D_2q^{\frac14})^n\Bigr).
\end{equation}
Substituting \eqref{eq22}--\eqref{eq24} into \eqref{eq5}, we obtain the assertion of the theorem.
\end{proof}

\begin{theorem}
\label{t4}
Let $p\equiv -3\!\pmod{8}$ and $2^m\parallel (q-1)$.  If $m=2$ then
\begin{align*}
N=\,&q^{n-1}+\frac{q-1}{4q}\biggl[\Bigl(-q^{\frac12}+q^{\frac14}i\sqrt{2(q^{\frac12}+C_1)}\,\Bigr)^n+\Bigl(-q^{\frac12}-q^{\frac14}i\sqrt{2(q^{\frac12}+C_1)}\,\Bigr)^n\biggr.\\
&+\biggl.\Bigl(q^{\frac12}+q^{\frac14}i\sqrt{2(q^{\frac12}-C_1)}\,\Bigr)^n+\Bigl(q^{\frac12}-q^{\frac14}i\sqrt{2(q^{\frac12}-C_1)}\,\Bigr)^n\biggr].
\end{align*}
If $m\ge 3$  then
\begin{align*}
N=\,&q^{n-1}+\frac{q-1}{2^mq}\cdot\biggl[2^{m-2}\cdot\Bigl((q^{\frac12}+2D_2q^{\frac14})^n+(q^{\frac12}-2D_2q^{\frac14})^n\Bigr)\biggr.\\
&\begin{aligned}
+\sum_{t=1}^{m-3}2^{m-t-2}\Bigl(&(-q^{\frac12}+\sum_{r=2}^t 2^{r-1}C_rq^{\frac{2^{r-1}-1}{2^{r}}}-2^t C_{t+1}q^{\frac{2^{t}-1}{2^{t+1}}}+2^{t+1} D_{t+2}q^{\frac{2^{t+1}-1}{2^{t+2}}})^n\Bigr.\\
+&(-q^{\frac12}+\sum_{r=2}^t 2^{r-1}C_rq^{\frac{2^{r-1}-1}{2^{r}}}-2^t C_{t+1}q^{\frac{2^{t}-1}{2^{t+1}}}-2^{t+1} D_{t+2}q^{\frac{2^{t+1}-1}{2^{t+2}}})^n\Bigl.\Bigr)\end{aligned}\\
&\begin{aligned}
+\Bigl(-q^{\frac 12}+\sum_{r=2}^{m-2}2^{r-1}C_r q^{\frac{2^{r-1}-1}{2^{r}}}&+2^{m-2}C_{m-1} q^{\frac{2^{m-2}-1}{2^{m-1}}}\\ &+2^{m-2}q^{\frac{2^{m-1}-1}{2^{m}}}i\sqrt{2(q^{\frac1{2^{m-1}}}-C_{m-1})}\,\Bigr)^n
\end{aligned}\\
\end{align*}
\begin{align*}
&\begin{aligned}
+\Bigl(-q^{\frac 12}+\sum_{r=2}^{m-2}2^{r-1}C_r q^{\frac{2^{r-1}-1}{2^{r}}}&+2^{m-2}C_{m-1} q^{\frac{2^{m-2}-1}{2^{m-1}}}\\ &-2^{m-2}q^{\frac{2^{m-1}-1}{2^{m}}}i\sqrt{2(q^{\frac1{2^{m-1}}}-C_{m-1})}\,\Bigr)^n
\end{aligned}\\
&\begin{aligned}
+\Bigl(-q^{\frac 12}+\sum_{r=2}^{m-2}2^{r-1}C_r q^{\frac{2^{r-1}-1}{2^{r}}}&-2^{m-2}C_{m-1} q^{\frac{2^{m-2}-1}{2^{m-1}}}\\ &+2^{m-2}q^{\frac{2^{m-1}-1}{2^{m}}}i\sqrt{2(q^{\frac1{2^{m-1}}}+C_{m-1})}\,\Bigr)^n
\end{aligned}\\
&\begin{aligned}
+\Bigl(-q^{\frac 12}+\sum_{r=2}^{m-2}2^{r-1}C_r q^{\frac{2^{r-1}-1}{2^{r}}}&-2^{m-2}C_{m-1} q^{\frac{2^{m-2}-1}{2^{m-1}}}\\&-2^{m-2}q^{\frac{2^{m-1}-1}{2^{m}}}i\sqrt{2(q^{\frac1{2^{m-1}}}+C_{m-1}\big)}\,\Bigr)^n\biggl.\biggr]. \end{aligned}
\end{align*}
The integers $C_r$ and $|D_r|$ are uniquely determined by~\eqref{eq21}.
\end{theorem}

\begin{proof}
Since $2^{m-2}\parallel s$, we conclude that
\begin{align*}
W_{1,m}(1)&=\begin{cases}
q^{1/2}&\text{if $m=2$,}\\
-q^{1/2}&\text{if $m\ge 3$,}
\end{cases}\\
W_{r,m}(1)&=2^{r-1}C_rq^{(2^{r-1}-1)/2^r},\quad 2\le r\le m-2,\quad m\ge 4,\\
W_{m-1,m}(1)&=(-1)^m\cdot 2^{m-2}C_{m-1}q^{(2^{m-2}-1)/2^{m-1}},\quad m\ge 3,\\
W_{m,m}(1)&=\pm 2^{m-2}q^{(2^{m-1}-1)/2^m}i\sqrt{2(q^{1/2^{m-1}}-(-1)^m C_{m-1})}\,.
\end{align*}
Therefore,
\begin{equation}
\label{eq25}
S_{m-1}+S_m=\Bigl(q^{\frac12}+q^{\frac14}i\sqrt{2(q^{\frac12}-C_1)}\,\Bigr)^n+\Bigl(q^{\frac12}-q^{\frac14}i\sqrt{2(q^{\frac12}-C_1)}\,\Bigr)^n
\end{equation}
if $m=2$, and
\begin{align}
&S_{m-1}+S_m\hskip-35pt&=
\Bigl(-q^{\frac 12}&+\sum_{r=2}^{m-2}2^{r-1}C_r q^{\frac{2^{r-1}-1}{2^{r}}}+(-1)^m\cdot 2^{m-2}C_{m-1} q^{\frac{2^{m-2}-1}{2^{m-1}}}\notag\\ 
&&&+2^{m-2}q^{\frac{2^{m-1}-1}{2^{m}}}i\sqrt{2(q^{\frac1{2^{m-1}}}-(-1)^m C_{m-1})}\,\Bigr)^n\notag\\
&&+\,\Bigl(-q^{\frac 12}&+\sum_{r=2}^{m-2}2^{r-1}C_r q^{\frac{2^{r-1}-1}{2^{r}}}+(-1)^m\cdot2^{m-2}C_{m-1} q^{\frac{2^{m-2}-1}{2^{m-1}}}\notag\\
&&&-2^{m-2}q^{\frac{2^{m-1}-1}{2^{m}}}i\sqrt{2(q^{\frac1{2^{m-1}}}-(-1)^mC_{m-1}\big)}\,\Bigr)^n\label{eq26}
\end{align}
if $m\ge 3$. Further,
\begin{align*}
W_{1,m-2}(1)&=-q^{1/2},\\
W_{r,m-2}(1)&=2^{r-1}C_rq^{(2^{r-1}-1)/2^r},\quad 2\le r\le m-2,\quad m\ge 4,\\
W_{m-1,m-2}(1)&=-(-1)^m\cdot 2^{m-2}C_{m-1}q^{(2^{m-2}-1)/2^{m-1}},\quad m\ge 3,\\
W_{m,m-2}(1)&=\pm 2^{m-2}q^{(2^{m-1}-1)/2^m}i\sqrt{2(q^{1/2^{m-1}}+(-1)^m C_{m-1})}\,.
\end{align*}
Hence
\begin{equation}
\label{eq27}
S_{m-2}=\Bigl(-q^{\frac12}+q^{\frac14}i\sqrt{2(q^{\frac12}+C_1)}\,\Bigr)^n+\Bigl(-q^{\frac12}-q^{\frac14}i\sqrt{2(q^{\frac12}+C_1)}\,\Bigr)^n
\end{equation}
if $m=2$, and
\begin{align}
&\,\,\,\,\,\,S_{m-2}\hskip-35pt&=
\Bigl(-q^{\frac 12}&+\sum_{r=2}^{m-2}2^{r-1}C_r q^{\frac{2^{r-1}-1}{2^{r}}}-(-1)^m\cdot 2^{m-2}C_{m-1} q^{\frac{2^{m-2}-1}{2^{m-1}}}\notag\\ 
&&&+2^{m-2}q^{\frac{2^{m-1}-1}{2^{m}}}i\sqrt{2(q^{\frac1{2^{m-1}}}+(-1)^m C_{m-1})}\,\Bigr)^n\notag\\
&&+\,\Bigl(-q^{\frac 12}&+\sum_{r=2}^{m-2}2^{r-1}C_r q^{\frac{2^{r-1}-1}{2^{r}}}-(-1)^m\cdot2^{m-2}C_{m-1} q^{\frac{2^{m-2}-1}{2^{m-1}}}\notag\\
&&&-2^{m-2}q^{\frac{2^{m-1}-1}{2^{m}}}i\sqrt{2(q^{\frac1{2^{m-1}}}+(-1)^mC_{m-1}\big)}\,\Bigr)^n\label{eq28}
\end{align}
if $m\ge 3$. By combining \eqref{eq26} and \eqref{eq28} and examining the two cases $m$ odd and $m$ even separately, we infer that for $m\ge 3$,
\begin{align*}
&S_{m-2}+S_{m-1}+S_m\hskip-20pt&=
\Bigl(-q^{\frac 12}&+\sum_{r=2}^{m-2}2^{r-1}C_r q^{\frac{2^{r-1}-1}{2^{r}}}+2^{m-2}C_{m-1} q^{\frac{2^{m-2}-1}{2^{m-1}}}\\ 
&&&+2^{m-2}q^{\frac{2^{m-1}-1}{2^{m}}}i\sqrt{2(q^{\frac1{2^{m-1}}}-C_{m-1})}\,\Bigr)^n\\
&&+\,\Bigl(-q^{\frac 12}&+\sum_{r=2}^{m-2}2^{r-1}C_r q^{\frac{2^{r-1}-1}{2^{r}}}+2^{m-2}C_{m-1} q^{\frac{2^{m-2}-1}{2^{m-1}}}\\
&&&-2^{m-2}q^{\frac{2^{m-1}-1}{2^{m}}}i\sqrt{2(q^{\frac1{2^{m-1}}}-C_{m-1}\big)}\,\Bigr)^n\\
&&+\Bigl(-q^{\frac 12}&+\sum_{r=2}^{m-2}2^{r-1}C_r q^{\frac{2^{r-1}-1}{2^{r}}}-2^{m-2}C_{m-1} q^{\frac{2^{m-2}-1}{2^{m-1}}}\\ 
&&&+2^{m-2}q^{\frac{2^{m-1}-1}{2^{m}}}i\sqrt{2(q^{\frac1{2^{m-1}}}+C_{m-1})}\,\Bigr)^n
\end{align*}
\begin{align}
&&+\,\Bigl(-q^{\frac 12}&+\sum_{r=2}^{m-2}2^{r-1}C_r q^{\frac{2^{r-1}-1}{2^{r}}}-2^{m-2}C_{m-1} q^{\frac{2^{m-2}-1}{2^{m-1}}}\notag\\
&&&-2^{m-2}q^{\frac{2^{m-1}-1}{2^{m}}}i\sqrt{2(q^{\frac1{2^{m-1}}}+C_{m-1}\big)}\,\Bigr)^n\label{eq29}.
\end{align}
It is readily seen that for $m\ge 3$ the sums $S_0,\dots,S_{m-3}$ are determined by \eqref{eq23} and \eqref{eq24}. Substituting \eqref{eq23}, \eqref{eq24},  \eqref{eq25}, \eqref{eq27},  \eqref{eq29} into \eqref{eq5}, we deduce the desired result.
\end{proof}

\section{Numerical results}
\label{s5}

The theoretical results of this paper are supported by numerical experiments. Some
numerical results are listed in Table~1.

\begin{table}[t]
\caption{Numerical results.}
\begin{tabularx}{\textwidth}{XXXXr|XXXXr}
\hline
$p$&$s$&$m$&$n$&$N\quad\quad$&$p$&$s$&$m$&$n$&$N\quad\quad$\\
\hline
3&4&3&3&7041&5&2&2&5&498625\\
3&4&3&4&1130241&5&2&3&4&12289\\
3&4&3&5&41304321&5&2&3&5&129025\\
3&4&4&3&20481&5&4&2&3&416833\\
3&4&4&4&81921&5&4&2&4&250892929\\
3&4&4&5&126033921&5&4&3&3&94849\\
3&8&3&3&30805761&5&4&3&4&304182529\\
3&8&4&3&42298881&5&4&4&3&319489\\
3&8&5&3&167936001&5&4&4&4&369328129\\
\hline
\end{tabularx}
\end{table}

\section{Concluding remarks}
\label{s6}

The results of the previous sections can be applied to some other diagonal equations. As before, $2^m\mid(q-1)$, $N=N[x_1^{2^m}+\dots+x_n^{2^m}=0]$ and $\lambda$ denotes a character of order $2^m$ on $\mathbb F_q$.

Granville, Li and Sun~\cite{GLS} have shown that
$$
N[a_1^{}x_1^{d_1}+\dots+a_n^{}x_n^{d_n}=0]=N[a_1^{}x_1^{w_1}+\dots+a_n^{}x_n^{w_n}=0],
$$
where $w_j=\gcd(d_j,\text{\rm lcm}(d_1,\dots,d_{j-1},d_{j+1},\dots,d_n))$, $1\le j\le n$. Thus, if $h_1,\dots,h_n$ are pairwise coprime positive integers with $2^mh_1\cdots h_n\mid(q-1)$, then
$$
N[x_1^{2^mh_1}+\dots+x_n^{2^m h_n}=0]=N[x_1^{2^m}+\dots+x_n^{2^m}=0]=N,
$$
and so our formulas are valid for some more general equations.

Let $u_1>2,\dots,u_t>2$ be pairwise coprime odd positive integers with $u_j\mid(q-1)$ for all $j$. Assume in addition that for each $j\in\{1,\dots,t\}$ there exists a positive integer $\ell_j$ such that $u_j\mid(p^{\ell_j}+1)$, with $\ell_j$ chosen minimal. It follows from \cite[Theorem~11.6.2]{BEW} that $2\ell_j\mid s$ for all $j$. Cao and Sun~\cite{CS} obtained the factorization formulas for the number of solutions to diagonal equations. Combining their result with \cite[Corollary~4]{W1}, we infer that
\begin{align*}
N&[x_1^{2^m}+\dots+x_n^{2^m}+y_{11}^{u_1}+\dots+y_{1n_1}^{u_1}+\dots+y_{t1}^{u_t}+\dots+y_{tn_t}^{u_t}=0]\\
&=q^{n+n_1+\dots+n_t-1}+(-1)^{\sum_{j=1}^t ((s/\ell_j)-1)n_j}(N-q^{n-1})q^{(n_1+\dots+n_t)/2}\\
&\hphantom{=}\times\prod_{j=1}^t \frac{(u_j-1)^{n_j}+(-1)^{n_j}(u_j-1)}{u_j}.
\end{align*}

Now let $k\ge 2$ be even and $b_1,\dots,b_k\in\mathbb F_q^*$. Lemma~\ref{l1} yields
\begin{align*}
N&[x_1^{2^m}+\dots+x_n^{2^m}+b_1^{}y_1^2+\dots+b_k^{}y_k^2=0]\\
&=q^{n+k-1}+\frac{q-1}q\sum_{\substack{1\le j_1,\dots,j_n\le 2^m-1\\j_1+\dots+j_n\equiv 0\!\!\pmod{2^m}}}\eta(b_1\cdots b_k)G(\eta)^k G(\lambda^{j_1})\cdots G(\lambda^{j_n}).
\end{align*}
Since, by Lemma~\ref{l2}(a), $G(\eta)^2=\eta(-1)q$, we deduce that
$$
N[x_1^{2^m}+\dots+x_n^{2^m}+b_1^{}y_1^2+\dots+b_k^{}y_k^2=0]=q^{n+k-1}+\eta((-1)^{k/2}b_1\cdots b_k)q^{k/2}(N-q^{n-1}).
$$
In particular,
\begin{align*}
N&[x_1^{2^m}+x_2^{2^m}+b_1^{}y_1^2+\dots+b_k^{}y_k^2=0]\\
&=q^{k+1}+\eta((-1)^{k/2}b_1\cdots b_k)q^{k/2}(q-1)\cdot\begin{cases}
2^m-1&\text{if $2^{m+1}\mid(q-1)$,}\\
-1&\text{if $2^m\parallel(q-1)$,}
\end{cases}
\end{align*}
which is a special case of a result of Sun~\cite{S}.

Finally, we notice that in the more general case where $f$ is a nondegenerate quadratic form over $\mathbb F_q$ in an even number $k$ of variables, we have
$$
N[x_1^{2^m}+\dots+x_n^{2^m}+f(y_1,\dots,y_k)=0]=q^{n+k-1}+\eta((-1)^{k/2}\Delta)q^{k/2}(N-q^{n-1}),
$$
where $\Delta$ denotes the determinant of $f$.

\section*{Acknowledgment}
The author thanks the referee for a careful reading of the manuscript and  helpful suggestions.


\begin{thebibliography}{00}
\bibitem{B1}
I.~Baoulina, Generalizations of the Markoff-Hurwitz equations over finite fields, {\it J.~Number Theory} {\bf 118}~(2006) 31--52.

\bibitem{B2}
I.~Baoulina, On the number of solutions to the equation $(x_1+\cdots +x_n)^2=ax_1\cdots x_n$ in a finite field, {\it Int. J. Number Theory} {\bf 4}~(2008) 797--817.

\bibitem{B3}
I.~Baoulina, On the number of solutions to certain diagonal equations over finite fields, {\it Int. J. Number Theory} {\bf 6}~(2010) 1--14.

\bibitem{BEW}
B.~C.~Berndt, R.~J.~Evans and K.~S.~Williams, {\it Gauss and Jacobi Sums}
(Wiley-Interscience, New York, 1998).

\bibitem{B}
 R.~F.~Beyl, Cyclic subgroups of the prime residue group, {\it Am. Math. Mon.} {\bf 84}~(1977) 46--48.

\bibitem{CS}
W.~Cao and Q.~Sun, Factorization formulae on counting zeros of diagonal equations over finite fields, {\it Proc. Am. Math. Soc.}
{\bf 135}~(2007) 1283--1291.

\bibitem{GLS}
A.~Granville, S.~Li and Q.~Sun, On the number of solutions of the equation $\sum_{i=1}^n x_i/d_i\equiv 0\!\!\pmod{1}$ and of diagonal equations in finite fields, {\it Sichuan Daxue Xuebao} {\bf 32}~(1995) 243--248.

\bibitem{J}
J.~R.~Joly, Nombre de solutions de certaines \'equations diagonales sur un corps fini, {\it C.~R.~Acad. Sci. Paris Ser.~A--B} {\bf 272}~(1971) 1549--1552.

\bibitem{KR}
S.~A.~Katre and A.~R.~Rajwade, Resolution of the sign ambiguity in
the determination of the cyclotomic numbers of order $4$ and the
corresponding Jacobsthal sum, {\it Math. Scand.} {\bf 60}~(1987) 52--62.

\bibitem{LN}
R.~Lidl and H.~Niederreiter, {\it Finite Fields} (Addison-Wesley, Reading, MA, 1983).

\bibitem{MP}
G.~L.~Mullen and D.~Panario, {\it Handbook of Finite Fields}  (CRC Press, 2013).

\bibitem{S}
Q.~Sun, On diagonal equations over finite fields, {\it Finite Fields Appl.} {\bf 3}~(1997) 175--179.

\bibitem{SY}
Q.~Sun and P.-Z.~Yuan, On the number of solutions of diagonal equations over a finite field, {\it Finite Fields Appl.} {\bf 2}~(1996) 35--41.

\bibitem{W}
A.~Weil, Numbers of solutions of equations in finite fields, {\it Bull. Am. Math. Soc.} {\bf 55}~(1949) 497--508.

\bibitem{W1}
J.~Wolfmann, The number of solutions of certain diagonal equations over finite fields, {\it J.~Number Theory} {\bf 42}~(1992) 247--257.

\bibitem{W2}
J.~Wolfmann, New results on diagonal equations over finite fields from cyclic codes, {\it Contemp. Math.} {\bf 168}~(1994) 387--395.

\end{thebibliography}
\end{document}